\documentclass[11pt,reqno]{amsart}
\usepackage{amsmath,amsthm,amssymb,enumerate}
\usepackage{hyperref}

\newcommand{\mysection}[1]{\section{#1}
      \setcounter{equation}{0}}
\usepackage{color}

\newtheorem{theorem}{Theorem}[section]
\newtheorem{lemma}[theorem]{Lemma}

\newtheorem{corollary}[theorem]{Corollary}

\theoremstyle{definition}
\newtheorem{assumption}{Assumption}[section]
\newtheorem{definition}{Definition}[section]
\newtheorem{example}{Example}[section]

\theoremstyle{remark}
\newtheorem{remark}{Remark}[section]

\newcommand\bG{\mathbb{G}}
\newcommand\bR{\mathbb{R}}

\newcommand\cB{\mathcal{B}}

\newcommand\cF{\mathcal{F}}

\newcommand\cP{\mathcal{P}}

\newcommand\cR{\mathcal{R}}

\makeatletter
 \newcommand{\sumstar}%
 {\operatornamewithlimits{\sum@\kern-.2em\raise1ex\hbox{*}}}
 \makeatother

\newcommand*{\beq}{\begin{equation}}
\newcommand*{\eeq}{\end{equation}}
\newcommand{\E}{\mathbb{E}}

\newcommand{\R}{\mathbb{R}}
\newcommand{\bit}{\begin{itemize}}
\newcommand{\eit}{\end{itemize}}
\DeclareMathOperator*{\esssup}{ess\,sup}

\begin{document}
\title[Finite difference schemes]{Finite difference schemes for stochastic partial 
differential equations in Sobolev spaces}
\author[M. Gerencs\'er]{M\'at\'e Gerencs\'er}
\address{Department of Mathematics and
Statistics, University of Edinburgh,
King's  Buildings,
Edinburgh, EH9 3JZ,
     United Kingdom}
\email{m.gerencser@sms.ed.ac.uk}

\author[I. Gy\"ongy]{Istv\'an Gy\"ongy}
\address{Department of Mathematics and
Statistics, University of Edinburgh,
King's  Buildings,
Edinburgh, EH9 3JZ,
     United Kingdom}
\email{i.gyongy@ed.ac.uk}


\begin{abstract}
We discuss $L_p$-estimates for finite difference schemes approximating parabolic, possibly degenerate, SPDEs, with initial conditions from $W^m_p$ and free terms taking values in $W^m_p.$ Consequences of these estimates include an asymptotic expansion of the error, allowing the acceleration of the approximation by Richardson's method.
\end{abstract}
\maketitle

\section{Introduction}
In this paper spatial finite difference schemes for parabolic 
stochastic partial differential equations (SPDEs) are considered. 
In the literature finite difference approximations for 
deterministic partial differential equations are well studied, 
we refer the reader to \cite{CL}, to recent results in \cite{DK}, 
and the references therein. There is a growing number of publications 
on finite difference schemes also for SPDEs, see e.g. \cite{DG}, 
\cite{GyM}, \cite{Y}, and their references.  In recent papers,  
see e.g. \cite{Gy2}, \cite{Gy0}, \cite{GyK}, \cite{H},  
$L_2$-theory is used to estimate in $W^m_2$-norms the 
error of  
finite difference approximations for the solutions of parabolic SPDEs. Hence error estimates in 
supremum norms are proved via Sobolev's embedding if $2m$ is larger 
than the dimension $d$ of the state space $\bR^d$.  
Therefore to get estimates in supremum norm, in these papers 
unnecessary spatial smoothness of the coefficients of the equation are required. 
Moreover, the smoothness conditions in these papers depend on the dimension 
of the state space. Our aim is to overcome this problem and generalize 
the results of \cite{Gy0} by giving $W^m_p$-norm estimates, 
assuming that the initial condition is in $W^m_p$ and the free terms 
are $W^m_p$-valued processes. 
This forces us to give up  part of generality, but important examples, 
like the Zakai equation in case of uncorrelated noises, are included. 
Since bounded functions, or more generally, functions with polynomial growth, 
can be seen as elements of suitable weighted Sobolev spaces 
with arbitrarily large integrability exponent $p$, for equations 
with such data we get dimension-invariant conditions 
on the smoothness of the coefficients. 

It should be noted 
that the $L_p$- and $L_q(L_p)$-theory of SPDEs are well developed, 
see e.g. \cite{Kim}, \cite{K2}, \cite{K3}. 
Their results, however, will not be used, as these theories
deal with uniformly parabolic SPDEs, while the equations 
in this paper may degenerate and become first order SPDEs.

Following the idea seen in \cite{K1}, to estimate the solutions 
of finite difference schemes we consider them  in
the whole space rather than on a grid. Through the estimates obtained 
for their Sobolev norms on the whole space, this allows us to estimate 
their supremum norm on a grid. For the finite difference approximations
not only their convergence is proved, but also power series expansion 
in the mesh size is obtained. As in \cite{GyK}, this allows us 
to accelerate the rate of convergence, using the well known 
Richardson extrapolation, introduced in \cite{R}.

Finally, let us introduce some notation used throughout the paper. 
We consider a complete probability space $(\Omega,\mathcal{F},P)$, 
which is  
equipped with a filtration $\mathbb{F}=(\mathcal{F}_t)_{t\geq0}$ and carries 
a sequence of independent $\cF_t$-Wiener martingales 
$(w^r)_{r=1}^{\infty}$. We use the notation $\cP$ for the $\sigma$-algebra 
of the predictable subsets of $\Omega\times[0,T]$.
It is assumed that $\mathcal{F}_0$ contains every $P$-zero set. 
For $p\geq 2$ and $m\geq 0$, $W^m_p$ denotes the Sobolev space 
with exponent $p$ and order $m$. For integer $m$, this is the space 
of functions whose generalized partial derivatives up to order $m$ are 
in $L_p$, for non integer real $m$, $W^m_p$ is 
a fractional Sobolev space, or, 
as often cited in the literature, Bessel potential space,  
for the definition we refer to \cite{T}. The Sobolev spaces of $l_2$-valued 
functions will be denoted by $W^m_p(l_2)$. 
We use the notation
$$
D_i=\frac{\partial}{\partial x^i}, \quad 
\partial_v=\sum_{i=1}^{d}v^{i}D_i
$$
for $v\in(v^1,...,v^d)\in \bR^d$, and 
$$
D^{\alpha}=D_1^{\alpha_1}D_2^{\alpha_2}....D^{\alpha_d}_d
$$
for multi-indices $\alpha=(\alpha_1,....,\alpha_d)\in\{0,1,....\}^d$ of length
$$
|\alpha|:=\alpha_1+\alpha_2+....+\alpha_d. 
$$
Derivatives are understood in the generalized sense unless otherwise noted.
The summation convention with respect to repeated indices 
is used thorough the paper, where it is not indicated otherwise.

The paper is organized as follows. 
Formulation of the problem and the statements of the main results 
are collected in Section \ref{section-formulation}. 
The appropriate estimate for the finite difference scheme is derived 
in Section \ref{section-estimate}, and it is used in the proof of the 
main results in Section \ref{section-proof}.

\section{Formulation of the results}\label{section-formulation}
We consider the SPDE
$$
du_t(x)=\{D_i(a_t^{ij}(x) D_j u_t(x))
+b_t^{i}(x)D_iu_t(x)
+c_t(x)u_t(x)+f_t(x)\}\,dt
$$
\begin{equation}                                                                              \label{SPDE}
+(\mu^{i r}_tD_iu_t(x)+\nu^r_t(x)u_t(x)+g^r(x))\,dw^r_t
\end{equation}
for $(t,x)\in[0,T]\times\R^d=:H_T$, with the initial condition
\begin{equation}                                                                            \label{SPDE-initial}
u_0(x)=\psi(x)\;\;x\in\R^d, 
\end{equation}
with the summation convention here and in the rest of the paper is used with respect to the repeated 
indices $i$, $j$ and $r$. 

The initial value $\psi$ is an $\mathcal{F}_0$-measurable 
random variable with values in $W^1_p$ for a fixed $p\geq2$. 
For all $i,j=1,2,...,d$ the coefficients $a^{ij}=a^{ji}$, $b^{i}$ and $c$ are real-valued 
$\mathcal{P}\times\mathcal{B}(\R^d)$-measurable bounded functions,  
and $\mu^{i}=(\mu^{i r})_{r=1}^{\infty}$ and  
$\nu=(\nu^{r})_{r=1}^{\infty}$ are $l_2$-valued 
$\mathcal{P}\times\mathcal{B}(\R^d)$-measurable bounded functions
on $\Omega\times H_T$. 
The free terms $f=(f_t)_{t\geq0}$ and $g=(g_t)_{t\geq0}$ are 
$W^1_p$-valued and $W^1_{p}(l_2)$
-valued adapted processes. 

Let $m\in[1,\infty)$ Set 
$$
F_{{m},p}(t)
=\left(\int_0^t|f_t|_{W^{{m}}_p}^pdt\right)^{1/p},
\quad 
G_{{m},p}(t)
=\left(\int_0^t|g_t|_{W^{{m}}_p(l_2)}^pdt\right)^{1/p},  
$$
and let $K>0$ be a constant. We make the following assumptions.

\begin{assumption}                                                                          \label{assumption 1}
The derivatives of the coefficients 
$b^{i}$ and $c$ in $x\in\bR^d$ up to order $\lceil m\rceil$, 
and the derivatives of $a^{ij}$ in $x$ up to order $\lceil m\rceil+1$ 
are functions, bounded by $K$. The $l_2$-valued functions 
$\mu^{i}$ and $\nu$ satisfy either of the following:

(i) their derivatives in $x$ up to order $\lceil m\rceil+1$ 
are functions, in magnitude bounded by $K$. 

(ii) $\mu=(\mu^i)_{i=1}^{d}=0$ 
and the derivatives of $\nu$ in $x$ up to order $\lceil m\rceil$ 
are functions, in magnitude bounded by $K$.
\end{assumption}

\begin{assumption}                                                                          \label{assumption 2}
Almost surely $\psi\in W^{m}_p$, and either

(i) $F_{{m},p}(T)+G_{{m+1},p}(T)<\infty$ (a.s.), or

(ii) $\mu=0$ and $F_{{m},p}(T)+G_{{m},p}(T)<\infty$ (a.s.).
\end{assumption}

\begin{assumption}                                                                          \label{assumption 3}
Almost surely the matrix valued function
$$
{\tilde a}^{ij}_t(x):={a}^{ij}_t(x)-\tfrac{1}{2}\mu_t^{i r}(x)\mu_t^{j r}(x), 
\quad i,j=1,...,d
$$ 
 is positive semidefinite for each $(t,x)\in H_T$.

\end{assumption}
The notion of (generalised) solution is defined as follows. 
\begin{definition}                                                                                \label{definition solution}
A $W^1_p$-valued adapted weakly continuous process $(u_t)_{t\in[0,T]}$ 
is a solution of $\eqref{SPDE}$-\eqref{SPDE-initial} on the interval 
$[0,\tau]$ for a stopping time $\tau\leq T$, if almost surely
$$
(u_t,\varphi)
=(\psi,\varphi)+\int_0^t\{-(a^{ij}_sD_ju_s,D_i\varphi)
+(b_s^{i}D_iu_s+c_su_s+f_s,\varphi)\}\,ds
$$
\begin{equation}                                                                                         \label{solution}
+\int_0^t(\mu^{i r}_sD_iu_s
+\nu^r_su_s+g^r_s,\varphi)\,dw^r_s,
\end{equation}
for all $t\in[0,\tau]$, $\varphi\in C^{\infty}_0(\R^d)$, 
where 
$(v,\varphi)$ denotes the integral
$$
\int_{\bR^d}v(x)\varphi(x)\,dx
$$
for functions $\varphi$ and $v$ on $\bR^d$, when 
$v\varphi\in L_1(\bR^d)$.
\end{definition}

Existence and uniqueness theorems for degenerate SPDEs 
are established in \cite{KR} and \cite{GGK}.  
We will need a slight generalization of these results, 
which will be proven at the end of Section \ref{section-estimate}.

\begin{theorem}                                                                             \label{theorem 1}
Let Assumptions \ref{assumption 1}, \ref{assumption 2} and 
\ref{assumption 3} hold. Then \eqref{SPDE}-\eqref{SPDE-initial} 
has a unique solution $u=(u_t)_{t\in[0,T]}$ on $[0,T]$. Moreover, 
$u$ is a $W^m_p$-valued weakly continuous process, 
it is strongly continuous with values in $W^{m-1}_p$, 
and for all $l\in[0,m]$ and $q>0$ 
\begin{equation}                                                                              \label{l}
E\sup_{t\in[0,T]}|u(t)|^q_{W^l_p}
\leq N(E|\psi|^q_{W^l_p}+EF^q_{l,p}(T)+EG^q_{l+\kappa,p}). 
\end{equation}
where $\kappa=0$ if $(\mu^i)=0$ and $\kappa=1$ otherwise, 
and $N$ is a constant depending only on $T$, $d$, $K$, $p$, and $m$.
\end{theorem}

While Theorem \ref{theorem 1} is stated for a general equation 
of the form \eqref{SPDE}-\eqref{SPDE-initial}, 
all of the subsequent results will only be proven under the restriction 
$mu=0$.

To introduce the finite difference schemes approximating 
$\eqref{SPDE}$ first let $\Lambda_0,\Lambda_1\subset\R^d$ 
be two finite sets, the latter being symmetric to the origin, 
and $0\in\Lambda_1\setminus\Lambda_0$. 
Denote 
$$
\Lambda=\Lambda_0\cup-\Lambda_0\cup\Lambda_1
$$ 
and $|\Lambda|=\sum_{\lambda\in\Lambda}|\lambda|$. 
On $\Lambda$ we make the following assumption.
\begin{assumption}                                                             \label{AsLambda}
If any subset $\Lambda'\subset\Lambda$ is linearly dependent, 
then $\Lambda'$ is linearly dependent over the rationals.
\end{assumption}
Let $\bG_h$ denote the grid
$$
\bG_h=\{h(\lambda_1+\ldots+\lambda_n):\lambda_i\in\Lambda,n=1,2,...\},
$$
for $h>0$, and define the finite difference operators
$$
\delta_{h,\lambda}\varphi(x)=(1/h)(\varphi(x+h\lambda)-\varphi(x))
$$
and the shift operators
$$
T_{h,\lambda}\varphi(x)=\varphi(x+h\lambda)
$$
for $\lambda\in\Lambda$ and $h\neq0$. 
Notice that $\delta_{h,0}\varphi=0$ 
and $T_{h,0}\varphi=\varphi$. 
For a fixed $h>0$ consider the finite difference equation
\begin{equation}                                                                  \label{differenceeq}
du_t^h(x)=(L_t^h(x)u_t^h(x)+f_t(x))\,dt
+(\nu^r_t(x)u_t^h(x)+g^r_t(x))\,dw_t^r,
\end{equation}
for $(t,x)\in [0,T]\times \bG_h$, 
with the initial condition 
\begin{equation}\label{differenceeq-initial}
u_0^h(x)=\psi(x)
\end{equation}
for $x\in \bG_h$, where
$$
L^h_t\varphi
=\sum_{\lambda\in\Lambda_0}
\delta_{-h,\lambda}(
\mathfrak{a}_h^{\lambda}\delta_{h,\lambda}\varphi)
+\sum_{\gamma\in\Lambda_1}\mathfrak{p}_h^{\gamma}\delta_{h,\gamma}
\varphi
+\sum_{\gamma\in\Lambda_1}\mathfrak{c}_h^{\gamma}T_{h,\gamma}
\varphi
$$
for functions $\varphi$ on $\bG_h$.
The coefficients $\mathfrak{a}_h^{\lambda}$, $\mathfrak{p}_h^{\gamma}$,  
and $\mathfrak{c}_h^{\gamma}$ are $\mathcal{P}\times\mathcal{B}(\R^d)$-measurable 
bounded functions on $\Omega\times[0,T]\times\R^d$, with values in $\R$, 
and $\mathfrak{p}^0_h=0$ is assumed. 
All of them are supposed to be defined for $h=0$ as well, 
and to depend continuously on $h$. 

Note that Assumption \ref{AsLambda} ensures that 
$\bG_h\cap B$ 
is finite for any bounded set $B\subset\R^d$. 
This condition is necessary for \eqref{differenceeq} 
to be useful from a practical point of view.

One can look for solutions 
of the above scheme in the space of adapted stochastic processes 
with values in $l_{p,h}$, the space of real functions $\phi$ on 
$\bG_h$ such that
$$
|\phi|_{l_{p,h}}^p=\sum_{x\in \bG_h}|\phi(x)|^ph^d<\infty.
$$
The similar space is defined for $l_2$-valued functions and will be denoted by 
$l_{p,h}(l_2)$. For a fixed $h$ equation \eqref{differenceeq} 
is an SDE in $l_{p,h}$, 
with Lipschitz coefficients, by the boundedness of 
$\mathfrak{a}_h^{\lambda},\mathfrak{p}_h^{\gamma}$, 
 $\mathfrak{c}_h^{\gamma}$, and $\nu^r$. 
Hence if almost surely
$$
|\psi|_{l_{p,h}}^p+\int_0^T|f_t|_{l_{p,h}}^p+|g_t|_{l_{p,h}(l_2)}^p\,dt<\infty,
$$
then \eqref{differenceeq}-\eqref{differenceeq-initial} 
admits a unique $l_{p,h}$-valued solution $(u^h_{t})_{t\in[0,T]}$.
\begin{remark}
By well-known results on Sobolev embeddings, if $m>k+d/p$, 
there exists a bounded operator $J$ from $W^m_p$ 
to the space of functions with bounded 
and continuous derivatives up to order $k$ 
such that $Jv=v$ almost everywhere. 
In the rest of the paper we will always identify functions 
with their continuous modifications if they have one, 
without introducing new notation for them. 
It is also known, and can be easily seen, 
that if Assumption \ref{AsLambda} holds and 
$m>d/p$, then the for $v\in W^m_p$ the restriction of $Jv$ 
onto the grid $\bG_h$ is in $l_{p,h}$, moreover, 
\begin{equation}                                                      \label{embedding}
|Jv|_{l_{p,h}}\leq C|v|_{W^m_p},
\end{equation}
where $C$ is independent of $v$ and $h$.
\end{remark}
\begin{remark}
The $h$-dependency of the coefficients may seem artificial 
and in fact does not mean any additional difficulty in the proof 
of Theorems \ref{Thm1}-\ref{Thm3} below. 
However, we will make use of this generality to extend our results to 
the case when the data in the problem 
\eqref{SPDE}-\eqref{SPDE-initial} are in some weighted Sobolev spaces.
\end{remark}

Clearly
$$
\delta_{h,\lambda}\varphi(x)\rightarrow\partial_{\lambda}\varphi(x)
$$
as $h\rightarrow0$ for smooth functions $\varphi$, 
so in order to get that our finite difference operators 
approximate the corresponding differential operators, 
we make the following assumption.
\begin{assumption}                                                                              \label{As0}
We have, for every $i,j=1,\ldots,d$
\begin{equation}                                                                                  \label{may29}
a^{ij}
=\sum_{\lambda\in\Lambda_0}\mathfrak{a}_0^{\lambda}\lambda^i\lambda^j,
\end{equation}
\begin{equation}                                                                                    \label{may30}
b^{i}
=\sum_{\gamma\in\Lambda_1}\mathfrak{p}_0^{\gamma}\gamma^i,\,\,c
=\sum_{\gamma\in\Lambda_1}\mathfrak{c}_0^{\gamma},
\end{equation}
and for $P\times dt\times dx$-almost all $(\omega,t,x)$ 
we have
\begin{equation}                                                                               \label{nonnegativity}
\mathfrak{a}_h^{\lambda}\geq0,\,\,\,\mathfrak{p}_h^{\gamma}\geq0
\quad\text{for every 
$\lambda\in\Lambda_0$, $\gamma\in\Lambda_1$, $h\geq0$}.
\end{equation}
\end{assumption}

\begin{remark}
The restriction \eqref{may29} together with $\mathfrak{a}^{\lambda}_0\geq0$ 
is not too severe, we refer the reader to \cite{KFac} 
for a detailed discussion about matrix-valued functions 
which possess this property.
\end{remark}

\begin{example}                                                                                \label{example}
Suppose that the matrix $(a^{ij})$ is diagonal. 
Then taking $\Lambda_0=\{e_i:i=1\ldots d\}$ 
and $\Lambda_1=\{0\}\cup\{\pm e_i: i=1\ldots d\}$, 
where $(e_i)$ 
is the standard basis in $\R^d$, one can set
$$
\mathfrak{a}_h^{e_{i}}=a^{ii},\,\mathfrak{p}_h^{e_i}
=b^i+\theta^i,\,\mathfrak{p}_h^{-e_i}
=\theta_i,\,
\mathfrak c_h^0=c,\,
\mathfrak{p}_h^{0}=\mathfrak c_h^{\pm e_i}=0,
$$
with any $\theta^i\geq\max(0,-b^i)$, $i=1\ldots d$.
\end{example}

\begin{example}                                                                   \label{example2}
Suppose that $(a^{ij})$ is a 
$\mathcal{P}\otimes\cB(\bR^d)$-measurable function of 
$(\omega,t, x)$,
with values in a closed bounded polyhedron in the set 
of symmetric non-negative $d\times d$ matrices, such that its 
first and second order derivatives in $x\in\bR^d$ are continuous 
in $x$ and are bounded by a constant $K$. Then it is shown 
in  \cite{KFac} that one can obtain a finite set $\Lambda_0\subset\bR^d$ 
and $\mathcal{P}\otimes\cB(\bR^d)$-measurable, bounded, 
nonnegative functions 
$\mathfrak{a}^{\lambda}_0$, $\lambda\in\Lambda_0$ 
 such that \eqref{may29} holds, and the first and second order derivatives 
 of $\mathfrak{a}^{\lambda}_0$ in $x$ are bounded by a constant 
 $N$ depending only on $K$, $d$ and the polyhedron.   
 Such situation arises in applications when, for example,  
 $(a^{ij}_t(x))$ is a 
diagonally dominant symmetric non-negative definite matrix 
for each $(\omega,t,x)$, 
which by definition means that 
$$
2a^{ii}_t(x)\geq\sum_{j=1}^d|a^{ij}_t(x)|, \quad \text{for all $i=1,2,..,d$,  and $(\omega,t,x)$},  
$$
and hence it clearly follows that $(a^{ij})$ takes values 
in a closed polyhedron in the set 
of symmetric non-negative $d\times d$ matrices. Clearly, this polyhedron  
can be chosen to be bounded if 
$(a^{ij})$ is a bounded function.  
Moreover, in the case $d=2$ explicit formulas are given in \cite{K4}  to represent 
diagonally dominant symmetric non-negative definite matrices $(a^{ij})$ 
in the form \eqref{may29}. 

The coefficients of the first and zero order terms, 
i.e., $\mathfrak{p}^{\gamma}_h$ and 
$\mathfrak{c}^{\gamma}_h$ can be chosen as in Example \ref{example}.

If $(a^{ij})$ does not depend on $x$, 
and it is a bounded $\cP$-measurable function of $(\omega,t)$ with values 
in the set of diagonally dominant symmetric non-negative definite matrices, 
then we can take 
$$
\Lambda_0:=\{e_i, e_i+e_j,e_i-e_j:i,j=1,2,...,d\}, 
\quad 
\Lambda_1=\{0\}\cup\Lambda_0\cup-\Lambda_0, 
$$ 
where $(e_i)_{i=1}^d$ is the standard basis in $\bR^d$, and set 
$$
\mathfrak{a}_h^{\lambda}=\left\{ \begin{array}{lll}
a^{ij}-\sum_{j\neq i}|a^{ij}| & \mbox{if $\lambda=e_i$}\\
\tfrac{1}{2}\sum_{j\neq i}(a^{ij})^{+} & \mbox{if $\lambda=e_i+e_j$}\\
\tfrac{1}{2}\sum_{j\neq i}(a^{ij})^{-} & \mbox{if $\lambda=e_i-e_j$}\\
\end{array} \right. , 
$$
$$
\mathfrak{p}_h^{\gamma}=\left\{ \begin{array}{lll}
\pm\tfrac{1}{2}b^{i}+\theta^i & \text{if $\gamma=\pm e_i$}\\
\theta^{ij} & \text{if $\gamma=\pm(e_i+e_j)$}\\
\theta^{ij} & \text{if $\gamma=\pm(e_i-e_j)$}\\
\end{array} \right. , 
$$
$$
\mathfrak c_h^0=c,
\quad
\mathfrak{p}_h^{0}=\mathfrak c_h^{\gamma}=0
\quad
\text{for $\gamma\in\Lambda_1\setminus\{0\}$}, 
$$
with any constants $\theta^{ij}\geq \kappa$ 
and $\theta^i \geq \kappa-\tfrac{1}{2}|b^i|$, 
for $i,j=1,...,d$, where $\kappa$ is any nonnegative constant, and 
$a^{\pm}:=(|a|\pm a)/2$ for $a\in\bR$. 
Then clearly, $\Lambda_0$, $\Lambda_1$, $\mathfrak{a}_h^{\lambda}$, 
$\mathfrak{p}_h^{\gamma}$ and $\mathfrak c_h$ satisfy  
Assumptions \ref{AsLambda}, \ref{As0} above, and 
Assumption \ref{As1} below.

\end{example}

Since the compatibility condition \eqref{may29}-\eqref{may30} 
will always be assumed, any subsequent conditions will be 
formulated for the coefficients in \eqref{differenceeq}, 
which then automatically imply the corresponding 
properties for the coefficients in \eqref{SPDE}.

\begin{assumption}                                                                     \label{As1}
The coefficients 
$\mathfrak{a}_h^{\lambda}$ 
(resp., 
$\mathfrak{p}_h^{\gamma}$, $\mathfrak{c}_h^{\gamma}$, $\mu$), 
and their partial derivatives in the variable $(h,x)$ 
up to order $\lceil m\rceil+1$ (resp., $\lceil m\rceil$) 
are functions bounded by $K$.
\end{assumption}

\begin{assumption}                                                                        \label{As2}
The initial value $\psi$ is in $W^{{m}}_p$, 
and the free terms $f$ and $g$ 
are $W^{{m}}_p$-valued and $W^{{m}}_p(l_2)$-valued processes, 
respectively, 
such that almost surely $F_{m,p}(T)+G_{m,p}(T)<\infty$.
\end{assumption}

We are now about to present the main results. 
The first three theorems correspond to similar results 
in the $L_2$ setting from \cite{Gy0}. 
The key role in their proof is played by Theorem \ref{theorem-Apr9} 
below, 
which presents an upper bound for the $W^m_p$ norms of the solutions 
to \eqref{differenceeq}-\eqref{differenceeq-initial}. 
After obtaining this estimate, Theorems \ref{Thm1} through \ref{Thm3} 
can be proved in the same fashion as their counterparts in the $L_2$ setting. 
Therefore, in Section \ref{section-proof} only a sketch of the proof 
will be provided in which we highlight the main differences; 
for the complete argument we refer to \cite{Gy0}.

\begin{theorem}                                                                         \label{Thm1}
Let $k\geq0$ be an integer and let Assumptions \ref{AsLambda} 
through \ref{As2} 
hold with $m>2k+3+d/p$. 
Then there are continuous random fields $u^{(1)},\ldots u^{(k)}$ 
on $[0,T]\times \R^d$, 
independent of $h$, such that almost surely
\begin{equation}                                                                       \label{expansion}
u^h_t(x)=\sum_{j=0}^k\frac{h^j}{j!}u^{(j)}_t(x)+h^{k+1}r_t^h(x)
\end{equation}
for $t\in[0,T]$ and $x\in \bG_h$, where $u^{(0)}=u$, $r^h$ 
is a continuous random field on $[0,T]\times \R^d$, and for any $q>0$
$$
\E\sup_{t\in[0,T]}\sup_{x\in \bG_h}|r_t^h(x)|^q
+\E\sup_{t\in[0,T]}|r_t^h|^q_{l_{p,h}}
\leq N(\E|\psi|_{W^m_p}^q+\E F_{m,p}^q(T)+\E G_{m,p}^q(T))
$$
with $N=N(K,T,m,p,q,d,|\Lambda|).$
\end{theorem}
Once we have the expansion above, 
we can use Richardson extrapolation to improve the rate of convergence. 
For a given $k$ set
\begin{equation}\label{may30b}
(c_0,c_1,\ldots,c_k)=(1,0,0,\ldots,0)V^{-1},
\end{equation}
where $V$ denotes the $(k+1)\times(k+1)$ 
Vandermonde matrix $V=(V^{ij})=(2^{-(i-1)(j-1)})$, and define
$$v^h=\sum_{i=0}^{k}c_iu^{h_i},$$
where 
$h_i=h/2^i$.

\begin{theorem}\label{Thm2}
Let $k\geq0$ be an integer and 
let Assumptions \ref{AsLambda} through \ref{As2} 
hold with $m>2k+3+d/p$. 
Then for every $q>0$ we have
$$
\E\sup_{t\in[0,T]}\sup_{x\in \bG_h}|u_t(x)-v^h_t(x)|^q
+\E\sup_{t\in[0,T]}|u_t-v^h_t|^q_{l_{p,h}}
$$
$$
\leq N(\E|\psi|_{W^m_p}^q+\E F_{m,p}^q(T)+\E G_{m,p}^q(T))
$$
with $N=N(K,T,m,k,p,q,d,|\Lambda|).$
\end{theorem}

\begin{theorem}                                                                                            \label{Thm3}
Let $(h_n)_{n=1}^{\infty}\in l_q$ be a nonnegative sequence for some $q\geq 1$. 
Let $k\geq0$ be an integer and let Assumptions \ref{AsLambda} through \ref{As2} 
hold with $m>2k+3+d/p$. 
Then for every $\varepsilon>0$ there exists a random variable $\xi_{\varepsilon}$ 
such that almost surely
$$
\sup_{t\in[0,T]}\sup_{x\in \bG_h}
|u_t(x)-v_t^h(x)|\leq\xi_{\varepsilon}h^{k+1-\varepsilon}
$$
for $h=h_n$.
\end{theorem}

\begin{remark}
We can use $h_i=h/n_i$, $i=1\ldots k$, with any set of different integers 
$n_i$, with $n_1=1$. Then changing the matrix 
$V$ to $\tilde{V}=(\tilde{V}^{ij})=(n_i^{-j+1})$ in \eqref{may30b}, 
Theorems \ref{Thm2}-\ref{Thm3} remain valid. 
The choice $n_i=i$, for example, yields a more coarse grid, 
and can reduce computation time.
\end{remark}

Choosing $p$ large enough, in some cases one can get rid of the term $d/p$ 
in the conditions of the theorems above, thus obtaining dimension-invariant conditions. 
To this end, first denote the function 
$\rho_s(x)=1/(1+|x|^2)^{s/2}$
 defined on $\R^d$ for all $s\geq0$. 
We say that a function $F$ on $\R^d$ has polynomial growth of order 
$s$ if the $L_{\infty}$ norm of $F\rho_s$ is finite. 
For any integer $m\geq0$, the set of functions on $\bR^d$ 
which have 
polynomial growth of order $s$ and whose 
derivatives up to order $m$ are functions and have polynomial growth 
of order $s$ is denoted by $P^m_s$, and its equipped with the norm
$$
\|F\|_{P^m_s}=|F\rho_s|_{W^m_{\infty}}<\infty.
$$
The similar space is defined for $l_2$-valued functions and is denoted by 
$P^m_s(l_2)$. Note that for any integers $m>k\geq0$, if $F\in P^m_s$, 
then its partial derivatives up to order $k$ exist in the classical sense 
and along with $F$ are continuous functions. 
The polynomial growth property of order $s$ for functions on 
$\bG_h$ can also be defined analogously, 
the set of such functions is denoted by $P_{h,s}$.

Let $s\geq0$ and $m$ be a nonnegative integer. Consider again the equation
$$
du_t(x)=(D_{i}a_t^{ij}(x) D_{j}u_t(x)
+b_t^{i}(x)D_{j}u_t(x)+c_t(x)u_t(x)+f_t(x))\,dt
$$
\begin{equation}                                                                           \label{SPDE-poly}
+(\nu^r_t(x)u_t(x)+g^r(x))\,dw^r_t
\end{equation}
for $(t,x)\in[0,T]\times\R^d$, with the initial condition
\begin{equation}\label{SPDE-poly-initial}
u_0(x)=\psi(x)\;\;x\in\R^d,
\end{equation}
where we keep all our measurability conditions from 
\eqref{SPDE}-\eqref{SPDE-initial}. 
However, instead of the integrability conditions on $\psi, f_t, g_t$, we 
now assume the following. 

\begin{assumption}                                                                                    \label{As5}
The initial data $\psi$ is an  
$\mathcal{F}_0\times\mathcal{B}(\R^d)$-measurable 
mapping from $\Omega\times\R^d$ to $\R$, 
such that $\psi\in P^m_s$ (a.s.). 
The free data $f$ and $g$ are 
$\mathcal{P}\times\mathcal{B}(\R^d)$-measurable mappings 
from $\Omega\times[0,T]\times\R^d$ to $\R$ and $l_2$, respectively. 
Moreover, almost surely $(f_t)$ is a $P^m_s$-valued process and  
$(g_t)$ is a $P^m_s(l_2)$-valued process, such that
$$
\big|\|f_t\|_{{P^m_s}}+\|g_t\|_{P^m_s(l_2)}\big|_{L_{\infty}[0,T]}<\infty.
$$
\end{assumption}

\begin{definition}
A $\mathcal{P}\times\mathcal{B}(\R^d)$-measurable 
mapping $u$ from 
$\Omega\times[0,T]\times\R^d$ to $\R$ 
such that $(u_t)_{t\in[0,T]}$ is almost surely a $P^1_s$-valued bounded process, 
is called a classical solution of \eqref{SPDE-poly}-\eqref{SPDE-poly-initial} 
on $[0,T]$,  
if almost surely $u$ and its first and second order partial derivatives in 
$x$ are continuous functions of 
$(t,x)\in[0,T]\times\R^d$, and almost surely 
$$
u_t(x)=\psi(x)+\int_0^t
[
D_{i}(a_s^{ij}(x) D_{j}u_s(x))
+b_s^{i}(x)D_{j}u_s(x)+c_s(x)u_s(x)+f_s(x)
]\,ds
$$
$$
+\int_0^t
[
\nu^r_s(x)u_s(x)+g^r_s(x)
]
\,dw^r_s
$$
for all $(t,x)\in[0,T]\times\bR^d$ 
for a suitable modification of the stochastic integral 
in the right-hand side of the equation. 
\end{definition}

 If $m\geq 1$, then as noted above the initial condition 
 and free terms are continuous in space. 
 This makes it reasonable to consider the finite difference scheme 
 \eqref{differenceeq}-\eqref{differenceeq-initial} as an approximation 
 for the problem \eqref{SPDE-poly}-\eqref{SPDE-poly-initial}. 

\begin{theorem}                                                                         \label{Thm4}
Let $k\geq0$ be integer, 
and let $\overline{s}>s\geq0$ be real numbers. 
Suppose that  Assumptions \ref{AsLambda} \ref{As0}, \ref{As1}, 
and \ref{As5} hold with $m>2k+3$.

\begin{enumerate}[(i)]  
\item Equation \eqref{SPDE-poly}-\eqref{SPDE-poly-initial} 
admits a unique $P^{m-1}_{\overline{s}}$-valued 
classical solution $(u_t)_{t\in[0,T]}$. 

\item For fixed $h$ the corresponding finite difference equation 
\eqref{differenceeq}-\eqref{differenceeq-initial} admits a unique 
$P_{h,\overline{s}}$-valued solution $(u^h_t)_{t\in[0,T]}$.

\item Suppose furthermore 
$\mathfrak{p}_h^{\gamma}\geq \kappa$ for $\gamma\in\Lambda_1$, 
for some constant $\kappa>0$, and
$$\Lambda_0\cup-\Lambda_0\subset\Lambda_1.$$ 
Then there are continuous random fields $u^{(1)},\ldots u^{(k)}$ on 
$[0,T]\times \R^d$, independent of $h$, such that almost surely
$$
u^h_t=\sum_{j=0}^k\frac{h^j}{j!}u^{(j)}_t(x)+h^{k+1}r_t^h(x)
$$
for $t\in[0,T]$ and $x\in \bG_h$, 
where $u^{(0)}=u$, $r^h$ is a continuous random field on 
$[0,T]\times \R^d$, and for any $q>0$ 
$$
\E\sup_{t\in[0,T]}\sup_{x\in \bG_h}|r_t^h(x)
\rho_{\overline{s}}(x)|^q
+\E\sup_{t\in[0,T]}|r_t^h\rho_{\overline{s}}|^q_{l_{p,h}}
$$
$$
\leq N\left(\E\|\psi\|_{P^m_s}^q
+\E\big|\|f_t\|_{{P^m_s}}+\|g_t\|_{P^m_s(l_2)}\big|_{L_{\infty}[0,T]}^q\right)
$$
with some $N=N(K,T,m,s,\overline{s},q,d,|\Lambda|,\kappa).$

\item Let $(h_n)_{n=1}^{\infty}\in l_q$ be a 
nonnegative sequence for some $q\geq 1$. 
Then for every $\varepsilon, M>0$ there 
exists a random variable $\xi_{\varepsilon,M}$ 
such that almost surely
$$\sup_{t\in[0,T]}\sup_{x\in \bG_h,|x|\leq M}|u_t(x)-v_t^h(x)|
\leq\xi_{\varepsilon,M}h^{k+1-\varepsilon}$$
for $h=h_n$.
\end{enumerate}
\end{theorem}
This theorem will be proved in Section \ref{section-proof}. 

\begin{remark}
Condition $\mathfrak{p}_h^{\gamma}\geq c$ in 
assertion (iii) of  the above theorem is harmless, similarly to the second part of \eqref{nonnegativity}. As seen in Examples \ref{example} 
and \ref{example2}, we can always 
satisfy this additional requirement by adding a sufficiently large 
constant to $\mathfrak{p}_h^{\gamma}$. 
\end{remark}

\mysection{Estimate on the finite difference scheme}                             \label{section-estimate}
First let us collect some properties of the finite difference operators. 
Throughout this section we consider a fixed $h>0$ 
and use the notation 
$u_{\alpha}=D^{\alpha}u$. It is easy to see that, 
analogously to the integration by parts,
\begin{equation}                                                                                \label{May 10.0}
\int_{\R^d}v(\delta_{h,\lambda}u)\,dx
=\int_{\R^d}(\delta_{h,-\lambda}v)u\,dx
=-\int_{\R^d}(\delta_{-h,\lambda}v)u\,dx,
\end{equation}

when $v\in L_{q/{q-1}}$ and $u\in L_q$ for some $1\leq q\leq\infty$, 
with the convention $1/0=\infty$ and $\infty/(\infty-1)=1$. 
The discrete analogue of the Leibniz rule can be written as
\begin{equation}                                                                                    \label{Leibniz}
\delta_{h,\lambda}(uv)
=u(\delta_{h,\lambda}v)+(\delta_{h,\lambda}u)(T_{h,\lambda}v). 
\end{equation}

Finally, we will also make use of the simple identities
\begin{equation}                                                               \label{Apr 18}
T_{h,\alpha}\delta_{h,\beta}u=\delta_{h,\alpha+\beta}u-\delta_{h,\alpha}u,
\end{equation}
\begin{equation}                                                                                   \label{Mar13.0}
vv_{\lambda}=(1/2)(\delta_{\lambda}(v^2)-h(\delta_{\lambda}v)^2)
\end{equation}
 and the estimate 
\begin{equation}                                                                                   \label{v}
|\delta_{h,\lambda}v|_{L_p}
\leq
 |\int_0^1\partial_{\lambda}v(\cdot+\theta h\lambda)\,d\theta |_{L_p}
 \leq |\lambda||v|_{W^1_p}
\end{equation}
valid for $p\in[1,\infty]$ and $v\in W^1_p$, $h\neq0$ and $\lambda\in\bR^d$.

\begin{lemma}                                                                             \label{Lemma-diff}
For any $p\geq2$, $\lambda\in\bR^d$, $h\neq0$ and  
real function $v$ 
on $\R^d$ we can write
$$
\delta_{h,\lambda}(|v|^{p-2}v)
=F^{h,\lambda}_p(v)\delta_{h,\lambda}v,
$$
where $F^{h,\lambda}_p(v)\geq0$, and for $p>2$, $q=p/(p-2)$ 
and for all $v\in L_p(\bR^d)$
 
\begin{equation}                                                                        \label{F}        
|F^{h,\lambda}_p(v)|^q_{L_q}\leq (p-1)^q|v|^p_{L_p}.  
\end{equation}

\end{lemma}
\begin{proof}

The derivative of the function $G(r)=|r|^{p-2}r$ is 
$$
G'(r)=(p-1)|r|^{p-2}\geq0, 
$$ 
so we have
$$
\delta_{h,\lambda}(|v|^{p-2}v)(x)
=(1/h)G((1-\theta)v(x)+\theta v(x+\lambda h))|_{\theta=0}^1
$$
$$
=\int_0^1G'((1-\theta)v(x)+\theta v(x+\lambda h))
\delta_{h,\lambda}v(x)\,d\theta
=F^{h,\lambda}_p(v)\delta_{h,\lambda}v(x).
$$
By Jensen's inequality and the convexity of the function $|r|^p$, 
$$
|F_p^{h,\lambda}(v)|^q_{L_q}
\leq (p-1)^q\int_{\bR^d}\int_0^1
\theta |v(x+\lambda h)|^p+(1-\theta)|v(x+\lambda h)|^p\,d\theta\,dx. 
$$
Hence \eqref{F} follows by Fubini's theorem 
and the shift invariance of the Lebesgue measure. 
\end{proof}

\begin{lemma}                                                                                   \label{lemma-estimate}
Let $m$ be a nonnegative integer, and let $\alpha$ be a multi-index of 
length $m$.  Then 
the following statements hold.
\begin{itemize}
\item[(i)] Let $\mathfrak a$ be a nonnegative function on $\bR^d$ such that its 
generalised derivatives up to order $m+1$ are functions, 
in magnitude bounded by a constant $K$. 
If $m\geq1$ then let the first order generalised derivatives of 
$\sigma:=\sqrt{\mathfrak a}$ be also functions, bounded by $K$. Then 
for $u\in W^m_p$, $p\in[2,\infty)$, $\lambda\in\bR^d$ and $h\neq0$ 
\begin{equation}                                                                           \label{I}
\int_{\bR^d}|D^{\alpha}u|^{p-2}D^{\alpha}u
D^{\alpha}\delta_{-h,\lambda}({\mathfrak a}\delta_{h,\lambda}u)\,dx  
\leq N|u|^p_{W^m_p}. 
\end{equation}
\item[(ii)] Let $\mathfrak p$ be a nonnegative function on $\bR^d$ 
such that its generalised derivatives up to order $m\vee 1$ are functions bounded by $K$. 
Let $p=2^k$ for an integer $k\geq1$. Then for  
$u\in W^m_p$, $\lambda\in\bR^d$ and $h>0$ 
\begin{equation}                                                                         \label{J}
\int_{\bR^d}|D^{\alpha}u|^{p-2}D^{\alpha}u
D^{\alpha}({\mathfrak p}\delta_{h,\lambda}u)\,dx  
\leq N|u|^p_{W^m_p}. 
\end{equation}
\end{itemize}
The constant $N$ in the above estimates depend only 
on $m$, $p$, $d$, $K$ and $|\lambda|$. 
\end{lemma}

\begin{proof}
Recall the notation $u_{\alpha}=D^{\alpha}u$. 
For real functions $v$ and $w$ 
defined on $\bR^d$ we write $v\sim w$ 
if their integrals over $\R^d$ are the same.
We use the notation $v\preceq w$ if $v=w+F$ 
with a function $F$ whose 
integral over $\R^d$ can be estimated by $N|u|^p_{W^m_p}$, 
where $N$ is a constant 
depending only on $m$, $K$, $p$, $d$ and $|\lambda|$.  
To prove \eqref{I} we consider first the case 
$m=0$. 
By \eqref{May 10.0} 
and Lemma \ref{Lemma-diff} 
$$
|u|^{p-2}u\delta_{-h\lambda}(\mathfrak{a}\delta_{h,\lambda}u)
\sim
-\delta_{h,\lambda}(|u|^{p-2}u)\mathfrak{a}\delta_{h,\lambda}u
$$
\begin{equation}                                                                               \label{c}
=-F^{h,\lambda}_p(u)\mathfrak{a}(\delta_{h,\lambda}u)^2\leq0,
\end{equation}
where $F$ is the functional obtained from Lemma \ref{Lemma-diff}. 
Consequently, 
\eqref{I} holds for $m=0$. 
Assume now $m\geq1$. Then it is easy to see that 
\begin{equation}                                                                                \label{i}
|D^{\alpha}u|^{p-2}D^{\alpha}u
D^{\alpha}\delta_{-h,\lambda}({\mathfrak a}\delta_{h,\lambda}u)\preceq I_1+I_2, 
\end{equation}
with 
$$
I_1:=|u_{\alpha}|^{p-2}u_{\alpha}\sum_{(\alpha',\alpha'')\in A}
\delta_{-h,\lambda}
D^{\alpha'}\mathfrak{a}D^{\alpha''}\delta_{h,\lambda}u
$$
$$
I_2:=|u_{\alpha}|^{p-2}u_{\alpha}
\delta_{-h,\lambda}(\mathfrak{a}\delta_{h,\lambda}u_{\alpha}), 
$$
where $A$ is the set of ordered pairs of multi-indices $(\alpha',\alpha'')$ such that 
$|\alpha'|=1$ and $\alpha'+\alpha''=\alpha$. 
By \eqref{May 10.0} 
and Lemma \ref{Lemma-diff} 
$$
I_1
\sim-2F^{h,\lambda}_p(u_{\alpha})\sigma\delta_{h,\lambda}u_{\alpha}
\sum_{(\alpha',\alpha'')\in A}D^{\alpha'}\sigma\delta_{h,\lambda}u_{\alpha''}
$$
\begin{equation}                                                                        \label{e1}
\leq 
\varepsilon F^{h,\lambda}_p(u_{\alpha})
{\mathfrak a}(\delta_{h,\lambda}u_{\alpha})^2
+\varepsilon^{-1}dK^2F_p^{h,\lambda}(u_{\alpha})(\delta_{h,\lambda}u_{\alpha''})^2 
\end{equation}
for every $\varepsilon>0$, where the the simple 
inequality $2yz\leq \varepsilon y^2+\varepsilon^{-1}z^2$ is used with 
$
y=\sigma\delta_{h,\lambda}u_{\alpha}
$
and  
$
z=\sum_{(\alpha',\alpha'')\in A}D^{\alpha'}\sigma\delta_{h,\lambda}u_{\alpha''}
$
.
Using \eqref{c} with $u_{\alpha}$ in place of $u$ we get 
\begin{equation*}                                                                          \label{cc}                                                             
I_2
\preceq 
-F_p^{h,\lambda}(u_{\alpha})\mathfrak{a}(\delta_{h,\lambda}u_{\alpha})^2. 
\end{equation*}
Combining this with \eqref{e1} with sufficiently small $\varepsilon$, 
from \eqref{i}
we obtain 
$$
I\preceq NF_p^{h,\lambda}(u_{\alpha})
\sum_{(\alpha',\alpha'')\in A}(\delta^h_{\lambda}u_{\alpha''})^2
$$
$$ 
\leq N|F^{h,\lambda}_p(u_{\alpha})|^{q}
+N|\sum_{(\alpha',\alpha'')\in A}(\delta^h_{\lambda}u_{\alpha''})^2|^{p/2},  
$$
with $q=p/(p-2)$, which gives \eqref{I}, due to the estimates \eqref{F} 
and \eqref{v}. 
 To prove \eqref{J} 
notice that for $p=2^k$ 
$$
J:=|D^{\alpha}u|^{p-2}D^{\alpha}u
D^{\alpha}({\mathfrak p}\delta_{h,\lambda}u)
=
(D^{\alpha}u)^{p-1}
D^{\alpha}({\mathfrak p}\delta_{h,\lambda}u)
$$
$$
\preceq 
(D^{\alpha}u)^{p-1}
{\mathfrak p}\delta_{h,\lambda}u_{\alpha}. 
$$
Hence we can repeatedly use \eqref{Mar13.0} 
and the nonnegativity of $\mathfrak p_h^{\lambda}$ to get
$$
u_{\alpha}^{p-1}\mathfrak{p}\delta_{h,\lambda}u_{\alpha}\leq (1/2) u_{\alpha}^{p-2}\mathfrak{p}
\delta_{h,\lambda}u_{\alpha}^2
$$
$$
\leq (1/4) u_{\alpha}^{p-4}\mathfrak{p}\delta_{h,\lambda}
u_{\alpha}^4\leq\cdots\leq(1/2^k)\mathfrak{p}
\delta_{h,\lambda}u_{\alpha}^{2^k}.
$$
By \eqref{May 10.0}, $\mathfrak{p}\delta_{h,\lambda}u_{\alpha}^{p}$ 
has the same integral over $\R^d$ as 
$\delta_{h,-\lambda}\mathfrak{p}u_{\alpha}^p$, 
and hence \eqref{J} follows, since 
$|\delta_{h,-\lambda}\mathfrak{p}|\leq K|\lambda|$ by \eqref{v}.  
\end{proof}

\begin{corollary}                                                                      \label{proposition-estimate}
Let $m\geq1$ be an integer and $p=2^k$ for some integer $k\geq1$, and  let 
Assumptions \ref{As1} and \ref{As2}, along with the condition \eqref{nonnegativity} 
be satisfied. 
Then for $u\in W_p^m$, 
$f\in W_p^m$, $g\in W_p^{m}(l_2)$  
and for all 
multi-indices $\alpha$  of length $|\alpha|\leq m$ we have
$$
\int_{\R^d}(p-1)|u_{\alpha}|^{p-2}u_{\alpha}(x)D^{\alpha}(L_t^hu(x)+f(x))
$$
$$
+(1/2)(p-1)(p-2)|u_{\alpha}|^{p-2}(x)|D^{\alpha}(\nu^r(x)u(x)+g^r(x))|^2dx
$$
\begin{equation}                                                                            \label{Mar13.1}
\leq N(|u|_{W_p^m}^p+|f|_{W_p^m}^p+|g|_{W_p^m(l_2)}^p)
\end{equation}
for $P\times dt$-almost all $(\omega,t)\in\Omega\times[0,T]$, 
where $N$ is a constant depending only on $d,p, m, |\Lambda|,$ and $K$.
\end{corollary}

\begin{proof}
Using the notation of the preceding proof, by H\"older's inequality 
$$
u^{p-1}_{\alpha}D^{\alpha}(\mathfrak{c}^{\gamma}_hT_{h,\gamma}u)+u^{p-1}_{\alpha}f_{\alpha}+u^{p-2}_{\alpha}|D^{\alpha}(\nu^ru+g^r)|^2\preceq N(|f|_{W_p^m}^p+|g|_{W_p^m(l_2)}^p).
$$ The remaining two terms are estimated in Lemma \ref{lemma-estimate}.
\end{proof}

The following is a stochastic version of Gronwall's lemma, 
for its proof we refer to \cite{Gy1}.
\begin{lemma}\label{Lemma-Gronwall}
Let $(y_t)_{t\in[0,T]}$, $(F_t)_{t\in[0,T]}$, and $(G_t)_{t\in[0,T]}$ 
be two nonnegative adapted processes, and let $(m_t)_{t\in[0,T]}$ 
be a continuous local martingale such that for a constant $N$ almost surely
$$
dy_t\leq N(y_t+F_t)dt+dm_t
$$
for all $t\in[0,T]$. Assume furthermore that for some $p\geq2$ almost surely
$$d\langle m\rangle_t\leq N(y_t^2+G_ty_t^{2-(2/p)})dt.$$
Then for every $q\geq0$ there exists a constant $C$, depending only on 
$N$, $q$, $p$, and $T$, such that
$$\E\sup_{t\leq T}y_t^q\leq C\E y_0^q+C\E\left(\int_0^TF_tdt\right)^q
+C\E\left(\int_0^TG_t^{p/2}dt\right)^q.
$$

\end{lemma}

Consider \eqref{differenceeq} without restricting it to the grid 
$\bG_h$, 
that is,
\begin{equation}                                                                  \label{differenceeq2}
du_t^h=(L_t^hu_t^h+f_t)dt+\nu^r_tu_t^h+g^r_t)dw_t^r, \quad t\in(0,T],\,x\in\bR^d
\end{equation}
with the initial condition
\begin{equation}                                                             \label{differenceeq2-initial}
u_0^h(x)=\psi(x),\quad x\in\bR^d.
\end{equation}
The solution of \eqref{differenceeq2}-\eqref{differenceeq2-initial} 
is understood in the spirit of Definition \ref{definition solution}.
\begin{definition}                                                                    \label{definition FDE}
An $L_p(\bR^d)$-valued continuous adapted process 
$(u^h_t)_{t\in[0,T]}$ is a solution to 
\eqref{differenceeq2}-\eqref{differenceeq2-initial} 
on $[0,\tau]$ for a stopping time $\tau\leq T$ if almost surely 
\begin{equation}
(u^h_t,\varphi)=(\psi,\varphi)+\int_0^t(L^h_su^h_s+f_s,\varphi)\,ds
+\int_0^t(\nu^r_su^h_s+g_s,\varphi)\,dw^r_s
\end{equation}
for all $t\leq\tau$ and $\varphi\in C_0^{\infty}(\bR^d)$. 
\end{definition}

Assumption \ref{As1} implies that the operators $u\rightarrow L_t^h u$ 
and $u\rightarrow (b^r_t u)_{r=1}^{\infty}$ 
are bounded linear operators from $W^m_p$ to $W^m_p$ and to $W^m_p(l_2)$, 
respectively, with operator norm uniformly bounded in $(t,\omega)$. 
Therefore if Assumption \ref{As2} is also satisfied, \eqref{differenceeq2} 
is a SDE in the space $W^m_p$ with Lipschitz continuous coefficients. 
As such, it admits a unique continuous solution.

\begin{theorem}                                                                                        \label{theorem-Apr9}
Let Assumptions \ref{As1} and \ref{As2} 
hold with $m\geq1$, and let condition \eqref{nonnegativity} 
be satisfied. 
Then \eqref{differenceeq2}-\eqref{differenceeq2-initial} 
has a unique continuous 
$W^m_p$-valued solution $(u_t^h)_{t\in[0,T]}$, 
and for each $q>0$ there exists a constant 
$N=N(d,q,p,m,K,T,|\Lambda|)$ such that
\begin{equation}                                                                                            \label{Mar14.0}
\E\sup_{t\leq T}|u_t^h|_{W^m_p}^q\leq N(\E|\psi|_{W^m_p}^q
+\E F_{m,p}^q(T)+\E G_{m,p}^q(T))
\end{equation}
for all $h>0$.
\end{theorem}
\begin{proof}
By the preceding argument, we need only prove estimate \eqref{Mar14.0}. 
First let $m$ amd $p$ be as in the conditions of 
Corollary \ref{proposition-estimate}, and fix a $q>1$.  
Let $\alpha$ be a multi-index such that $|\alpha|\leq m$.   
If we apply It\^o's formula to 
$|D^{\alpha}u^h|^p_{L_p}$ by Lemma 5.1 in \cite{K0}, 
one can notice that the term appearing 
in the drift is the left-hand side of \eqref{Mar13.1}, 
with $u^h$ in place of $u$. Using Corollary \ref{proposition-estimate} 
and summing over $|\alpha|\leq m$ we get
$$
d|u_t^h|^p_{W^m_p}\leq N(|u|_{W_p^m}^p
+|f|_{W_p^m}^p+|g|_{W_p^m}^p)\,dt+dm_t^h
$$
with some $N$ depending only on $p, m,$ $d$, $|\Lambda|$, and $K$, 
where
$$
dm_t^h
=(p-1)\int_{\R^d}|\partial_{\alpha}u_t^h|^{p-1}\partial_{\alpha}(\nu_t^r u_t^h+g_t^r)\,dx\,dw^r_t
$$
with $\alpha$ used as a repeated index. It is clear that
$$
d\langle m^h\rangle_t
=(p-1)^2\left(\int_{\R^d}|\delta_{\alpha}u_t^h|^{p-1}
\partial_{\alpha}(\nu_t^r u_t^h+g_t^r)dx\right)^2dt
$$
$$
\leq N((|u_t^h|^{p}_{W^m_p})^2+|g|_{W^m_p}^2|u_t^h|^{2p-2}_{W^m_p})\,dt,
$$
so Lemma \ref{Lemma-Gronwall} 
can be applied to the function $|u_t^h|^{p}_{W^m_p}$ and the power $q/p$, 
which proves \eqref{Mar14.0} for integer $m$, $p=2^k$.  

Note that \eqref{Mar14.0} is equivalent to
$$
[\E\sup_{t\leq T}|u_t^h|_{W^m_p}^q]^{\frac{1}{q}}
\leq N([\E|\psi|_{W^m_p}^q]^{\frac{1}{q}}+
[\E F_{m,p}^{q}]^{\frac{1}{q}}+[\E G_{m,p}^{q}]^{\frac{1}{q}}),
$$
which implies
\begin{equation}                                                                                        \label{Feb19} 
[\E\left(\int_0^T|u_t^h|_{W^m_p}^r\right)^{\frac{q}{r}}]^{\frac{1}{q}}
\leq N([\E|\psi|_{W^m_p}^q]^{\frac{1}{q}}
+
[\E F_{m,p}^{q}]^{\frac{1}{q}}+[\E G_{m,p}^{q}]^{\frac{1}{q}}),
\end{equation}
for any $r>1$, with another constant $N$, independent from $r$.
In other words, this means that for the special 
case of $m$ and $p$ considered so far the solution operator
$$
(\psi,f,g)\rightarrow u^h
$$
continuously maps 
$\Psi^m_p\times \mathcal{F}^m_p\times \mathcal{G}^m_p$ 
to $\mathcal{U}^m_p$, where
$$
\Psi^m_p=L_{q}(\Omega,W_p^m),$$
$$\mathcal{F}^m_p=L_{q}(\Omega,L_p([0,T],W_p^m)),$$
$$\mathcal{G}^m_p=L_{q}(\Omega,L_p([0,T],W_p^m(l_2))),$$
$$\mathcal{U}^m_p=L_q(\Omega,L_{r}([0,T],W^m_p)).$$
Let us denote the complex interpolation space 
between any Banach spaces $A_0$ and $A_1$ 
with parameter $\theta$ by $[A_0,A_1]_{\theta}$. 
Recall the following interpolation properties 
(see 1.9.3, 1.18.4, and 2.4.2 from \cite{T})
\begin{enumerate}[(i)]
\item 
If a linear operator $T$ is continuous from $A_0$ to $B_0$ and from $A_1$ to $B_1$, then it is also continuous from $[A_0,A_1]_{\theta}$ to $[B_0,B_1]_{\theta}$, moreover, its norm between the interpolated spaces depends only on $\theta$ and its norm between the original spaces.
\item 
For a measure space $M$ and $1< p_0,p_1<\infty$, 
$$
[L_{p_0}(M,A_0),L_{p_1}(M,A_1)]_{\theta}=L_{p_{\theta}}(M,[A_0,A_1]_{\theta}),
$$ 
where $1/p_{\theta}=(1-\theta)/p_0+\theta/p_1$.
\item 
For $m_0,m_1\in\bR$, $1<p_0,p_1<\infty$,
$$
[W^{m_0}_{p_0},W^{m_1}_{p_1}]_{\theta}=W^{m_{\theta}}_{p_{\theta}},
$$
where $m_{\theta}=(1-\theta)m_0+\theta m_1$, 
and $1/p_{\theta}=(1-\theta)/p_0+\theta/p_1$.
\end{enumerate}
Now take any $p\geq2$, and take $p_0\leq p\leq p_1$ 
such that $p_0=2^k$ and $p=2^{k+1}$ for $k\in\mathbb{N}$, 
and set $\theta\in[0,1]$ such that 
$1/p=(1-\theta)/p_0+\theta/p_1.$ 
By property (ii) we have
$$
\Psi^m_p=[\Psi^m_{p_0},\Psi^m_{p_1}]_{\theta},\;\mathcal{F}^m_p
=[\mathcal{F}^m_{p_0},\mathcal{F}^m_{p_1}]_{\theta},
$$
$$\mathcal{G}^m_p
=[\mathcal{G}^m_{p_0},\mathcal{G}^m_{p_1}]_{\theta},\;\mathcal{U}^m_p
=[\mathcal{U}^m_{p_0},\mathcal{U}^m_{p_1}]_{\theta},
$$
and therefore by (i) the solution operator is continuous for any $p\geq 2$  
and integer $m\geq1$.

For arbitrary $m\geq1$, $p\geq 2$, set $\theta=\lfloor m\rfloor/\lceil m\rceil$.   
Then properties (ii) and (iii) imply that
$$
\Psi^m_p=[\Psi^{\lfloor m\rfloor}_{p},\Psi^{\lceil m\rceil}_{p}]_{\theta},\;\mathcal{F}^m_p=[\mathcal{F}^{\lfloor m\rfloor}_{p},\mathcal{F}^{\lceil m\rceil}_{p}]_{\theta},
$$
$$
\mathcal{G}^m_p
=[\mathcal{G}^{\lfloor m\rfloor}_{p},\mathcal{G}^{\lceil m\rceil}_{p}]_{\theta},\;\mathcal{U}^m_p
=[\mathcal{U}^{\lfloor m\rfloor}_{},\mathcal{U}^{\lceil m\rceil}_{p}]_{\theta}.
$$
If Assumptions \ref{As1} and \ref{As2} hold, 
then the solution operator is continuous from 
$\Psi^{\lceil m\rceil}_{p}
\times 
\mathcal{F}^{\lceil m\rceil}_{p}
\times 
\mathcal{G}^{\lceil m\rceil}_p$ 
to $\mathcal{U}^{\lceil m\rceil}_{p}$, 
and from 
$\Psi^{\lfloor m\rfloor}_{p}
\times \mathcal{F}^{\lfloor m\rfloor}_{p}
\times \mathcal{G}^{\lfloor m\rfloor}_{p}$ 
to $\mathcal{U}^{\lfloor m\rfloor}_{p}$. 
Applying property (i) again therefore yields 
\eqref{Feb19} for $m$, $p$. 
Letting $r\rightarrow\infty$ and keeping in mind that 
$u^h$ 
is a continuous in $W^m_p$-valued process, 
using Fatou's lemma we get \eqref{Mar14.0} when $q>1$. 
 Hence for $q>1$ we obtain 
$$
E({\bf1}_A\sup_{t\leq \tau\wedge\tau_n}
|u_t^h|_l^q)
\leq NE|{\bf1}_A{\bf1}_{\tau_n>0}\psi|_{W^m_p}^q
+
N\E({\bf1}_AF_{m,p}^q(\tau\wedge\tau_n))
$$
$$
+N\E ({\bf1}_AG_{m,p}^q(\tau\wedge\tau_n))
$$
for every stopping time $\tau\leq T$,  integer $n\geq1$, and $A\in\cF_0$, where 
$$
\tau_n:=\inf\{t\in[0,T]: \cR_{m,p}(t)\geq n\}, 
$$
$$
\cR_{m,p}^p(t):=|\psi|^p_{W^m_p}+\int_0^t|f_s|^p_{W^m_p}+|g_s|^p_{W^m_p}\,ds, 
$$
and $N$ is a constant depending only on $K$, $T$, $m$, $q$, $d$ and $|\Lambda|$. 
By virtue of Lemma 3.2 from \cite{GK2003} this implies 
$$
\E(\sup_{t\leq \tau\wedge\tau_n}
|u_t^h|_{W^m_p}^q)
\leq N(\E|\psi|_{W^m_p}^q
+
\E F_{m,p}^q(\tau\wedge\tau_n)
+
\E G_{m,p}^q(\tau\wedge\tau_n))
$$
for any $q>0$ 
with a constant $N=N(K,T,p,d,m,|\Lambda|)$. We finish the proof 
by letting here $n\to\infty$.  

\end{proof}

\emph{Proof of Theorem \ref{theorem 1}.} 
  By rewriting the equation in the non-divergence form, 
  the theorem for integer $m$ follows from Theorem 2.1 in 
  \cite{GGK}. 
  Thus we need only prove it when $m$ is non integer. From \cite{GGK}, 
  we know that under the conditions of the theorem,
  \eqref{SPDE}-\eqref{SPDE-initial} admits a unique 
 solution $u$. 
 Moreover, it is $W^{\lfloor m\rfloor}_p$-valued, and
\begin{equation}                                                                               \label{june4}
\E\sup_{t\in[0,T]}|u|_{W^l_p}^q
\leq N(\E|\psi|_{W^{l}_p}^q
+\E F_{l,p}^q(T)+\E G_{l+\kappa},p^q(T))
\end{equation}
holds for $l=0,1,...,\lceil m\rceil$ and $q>0$, where $\kappa=0$ when 
$(\mu^i)=0$ and $\kappa=1$ when $(\mu^i)$ is not identically zero. 
Assume first that $q>1$. 
Then following the same interpolation arguments as above, 
we find that \eqref{june4} holds for all $l\in[0,\lceil m\rceil]$,
with $\esssup$ in place of $\sup$ on the left-hand side. 
Then by substituting $(1-\Delta)^{(m-1)/2}\phi$ 
in place of $\phi$ in \eqref{solution}, 
we obtain that for any $\phi\in C_0^{\infty}$, 
almost surely for all $t\in[0,T]$, 
$$
((1-\Delta)^{(m-1)/2}u_t,\phi)=((1-\Delta)^{(m-1)/2}\psi,\phi)
$$
$$+\int_0^t\sum_{i=1}^d(\mathbf{F}^i_s,D_i\phi)
+(\mathbf{F}^0_s,\phi)\,ds+\int_0^t(\mathbf{G}^k_s,\phi)\,dw^k_s,
$$
where, due to estimates such as
$$
|(1-\Delta)^{(m-1)/2}(a^{ij}D_jv)|_{L_p}\leq K|v|_{W^m_p},
$$
 $\mathbf{F}^i$ and $\mathbf{G}=(\mathbf{G}^k)_{k=1}^{\infty}$ 
are predictable processes with values in $L_p$, such that
$$
\int_0^T(\sum_{i=0}^d|\mathbf{F}^i_t|_{L_p}
+|\mathbf{G}_t|_{L_p(l_2)})dt<\infty.
$$
Using It\^o's formula for the $L_p$ norm from \cite{K0}, 
we find that $(1-\Delta)^{(m-1)/2}u$ is a strongly $L_p$-valued process, 
and thus $u$ is a strongly continuous $W^{m-1}_p$-valued process. 
Hence almost surely 
\begin{equation}                                                                         \label{1.6.7.14}
((1-\Delta)^{m/2}u_t,\varphi)=((1-\Delta)^{(m-1)/2}u_t,(1-\Delta)^{1/2}\varphi)
\end{equation}
is continuous in $t\in[0,T]$ for all $\varphi\in C_0^{\infty}$. 
Let $\Phi$ denote the set of those $C^{\infty}_0$ functions which belong 
to the unit ball of $L_{p^{\ast}}$, where $p^{\ast}=p/(p-1)$. Then 
$$
\sup_{t\in[0,T]}|u_t|_{W^m_p}=\sup_{t\in[0,T]}\sup_{\varphi\in\Phi}
((1-\Delta)^{m/2}u_t,\varphi)=\sup_{\varphi\in\Phi}
\sup_{t\in[0,T]}((1-\Delta)^{m/2}u_t,\varphi)
$$
$$
=\sup_{\varphi\in\Phi}\esssup_{t\in[0,T]}((1-\Delta)^{m/2}u_{t},\varphi_j)
\leq 
\esssup_{t\in[0,T]}|u_{t}|_{W^m_p}<\infty\quad (a.s.). 
$$
This, 
the continuity in $t\in[0,T]$ of the expression in 
\eqref{1.6.7.14} and the denseness of $C_0^{\infty}$ in 
$W^{-m}_{p^{\ast}}$ imply
 that almost surely $u$ is a $W^m_p$-valued  
weakly continuous process.  Consequently, \eqref{june4} holds 
for all $l\in[0,m]$ and $q>1$. 
Hence using Lemma 3.2 from \cite{GK2003} in the same way 
as at the end of the proof 
of Theorem \ref{theorem-Apr9}, 
we obtain \eqref{june4} for all $l\in[0,m]$ and $q>0$.

\section{Proof of the main results}\label{section-proof}

\emph{Proof of Theorems \ref{Thm1}-\ref{Thm3}.} 
To prove Theorem \ref{Thm1}, first consider the functions
$$
F(h)=\delta_{h,\lambda}\phi(x)=\int_0^1\partial_{\lambda}\phi(x+h\theta\lambda)
\,d\theta,
$$
$$
G(h)=\delta_{h,\lambda}\delta_{h,\lambda}\psi(x)=\int_0^1\int_0^1\partial_{\lambda}\partial_{\lambda}\psi(x+h\lambda(\theta_1+\theta_2))d\theta_1d\theta_2
$$
for fixed $\phi\in W^{n+l+2}_p,\psi\in W^{n+l+3}_p$, $n,l\geq0$. Applying Taylor's formula at $h=0$ up to $n+1$ terms we get that
$$
|\delta_{h,\lambda}\phi-\sum_{i=0}^nh^iA_i\partial_{\lambda}^{i+1}\phi|_{W^l_p}\leq N|h|^{n+1}|\phi|_{W^{n+l+2}_p},
$$
$$
|\delta_{-h,\lambda}\delta_{h,\lambda}\psi
-\sum_{i=0}^nh^i
B_{i}\partial_{\lambda}^{i+2}\psi|_{W^l_p}
\leq N|h|^{n+1}|\psi|_{W^{l+n+3}_p}
$$
with constants $A_i=1/(i+1)!$ and 

\begin{equation*}                                                                      \label{1.8.7.14}
B_i= \left\{ \begin{array}{ll}
0 & \mbox{if $i$ is odd}\\
\frac{2}{(i+2)!} & \mbox{if $i$ is even}
\end{array} \right. , 
\end{equation*}
where $N=N(|\Lambda|,d,l,n)$ is a constant.
 Similarly, or in fact  equivalently to the first inequality, we have 
$$
|T_{h,\lambda}\varphi-\sum_{i=0}^n\frac{h^i}{i!}\partial_{\lambda}^i\varphi|_{W^l_p}
\leq N|h|^{n+1}|\varphi|_{W^{n+l+1}_p}
$$
for $\varphi\in W^{n+l+1}_p$, where $\partial^0_{\lambda}$ denotes the identity operator. Without going into details, 
it is clear that, due to Assumption \ref{As1},
from these expansions one can obtain operators $\mathfrak{L}_t^{(i)}$ 
for integers $i\in[0,\lceil m\rceil]$
such that
$\mathfrak{L}^0_t\phi
=\partial_i a^{ij}\partial_j\phi+b^i\partial_i\phi+c\phi$,
 
\begin{align}                                                                              
|\mathfrak{L}_t^{(i)}\phi|_{W^l_p}&\leq N|\phi|_{W^{l+i+1}_p} 
\quad\text{for $i$ odd, $i+l\leq \lceil m\rceil$},                                  \label{odd}          \\
|\mathfrak{L}_t^{(i)}\phi|_{W^l_p}&\leq N|\phi|_{W^{l+i+2}_p}           \label{even}
\quad\text{for $i$ even, $i+l\leq \lceil m\rceil$}, 
\end{align}

 and
\begin{equation}                                                   \label{extrapolation-estimate}
|(L^h_t-\sum_{i=0}^n\frac{h^i}{i!}\mathfrak{L}_t^{(i)})\phi|_{W^l_p}
\leq N|h|^{n+1}|\phi|_{W^{n+l+3}_p}
\quad\text{for $n+l<\lceil m\rceil$}
\end{equation}
with $N=N(|\Lambda|,K,d,p,m)$. 
The random fields $u^{(j)}$  in expansion \eqref{expansion}
can then be obtained from the system of SPDEs 
\begin{align}                                                                       
du^{(j)}_t&=(\mathfrak{L}^{(0)}_tu_t^{(j)}           
+\sum_{l=1}^j\tbinom{j}{l}\mathfrak{L}_t^{(l)}u_t^{(j-l)})\,dt
+\nu^r_t(x)u^{(j)}_t(x)\,dw^r_t                                              \label{May13.0}\\
u^{(j)}_0&=0, \quad j=1,...,k ,                                                \label{systemini}
\end{align}
where 
$v^{(0)}=u$, the solution of 
\eqref{SPDE}-\eqref{SPDE-initial} when $\mu=0$.
The following theorem holds, being the exact analogue 
of Theorem 5.1 from \cite{Gy0}. 
It can be proven inductively on $j$, 
by a straightforward application of 
Theorem \ref{theorem 1} and \eqref{odd}-\eqref{even}.

\begin{theorem}                                                                  \label{theorem system}
Let $k\geq1$ be an integer, and let  
Assumptions \ref{As0}, \ref{As1} and \ref{As2} hold with $m\geq2k+1$.
Then there is a unique solution $u^{(1)},\ldots, u^{(k)}$ 
of \eqref{May13.0}-\eqref{systemini}.  
Moreover, $u^{(j)}$ is a $W^{m-2j}_p$-valued weakly continuous process, 
it is
strongly continuous as a $W^{m-2j-1}_p$-valued process, and
$$
\E\sup_{t\in[0,T]}|u^{(j)}_t|^q_{W^{m-2j}_p}
\leq  N(\E|\psi|_{W^{m}_p}^q
+\E F_{m,p}^q(T)+\E G_{m,p}^q(T))
$$
for $j=1,\ldots,k$, for any $q>0$, with a constant $N=N(K,m,p,q,T,|\Lambda|)$.
\end{theorem}
Set
$$
\overline{r}^h_t(x)=u^h_t(x)
-\sum_{j=0}^k\frac{h^j}{j!}u_t^{(j)}(x), 
$$
for $t\in[0,T]$ and $x\in\bR^d$, where 
$u^h$ is the solution of \eqref{differenceeq2}-\eqref{differenceeq2-initial}. 
Then it is not difficult to verify that $\overline{r}^h$ is the solution, 
in the sense of Definition \ref{definition FDE}, of the finite difference equation 
\begin{equation*}    
\overline{r}^h_t(x)=(L^h_t\overline{r}^h_t(x)
+F^h_t(x))\,dt
+\nu^r_t(x)\overline{r}^h_t(x)\,dw^r_t, 
\quad t\in(0,T],\,x\in\bR^d
\end{equation*}  
with initial condition                                                    
$\overline{r}^h_0(x)=0$ for $x\in\bR^d$,                               
where
$$
F^h_t
=\sum_{j=0}^k\frac{h^j}{j!}
\left(L^h_t-\sum_{i=0}^{k-j}\frac{h^i}{i!}\mathfrak{L}_t^{(i)}\right)u^{(j)}_t.
$$
Hence by applying Theorem \ref{theorem-Apr9} we get
$$
\E\sup_{t\in[0,T]}|\overline{r}_t^h|_{W_p^{m-2k-3}}^q
\leq N\E\left(\int_0^t|F_t|_{W^{m-2k-3}_p}^pdt\right)^{q/p}.
$$
Now using $m-2k-3>d/p$, for the left-hand side we can write
$$
\E\sup_{t\in[0,T]}
\sup_{x\in {\bG_h}}
|\overline{r}_t^h(x)|^q+\E\sup_{t\in[0,T]}
|\overline{r}_t^h|^q_{l_{p,h}}\leq N\E\sup_{t\in[0,T]}
|\overline{r}_t^h|_{W_p^{m-2k-3}}^q,
$$
while \eqref{extrapolation-estimate} and the theorem above yield
$$
\E\sup_{t\in[0,T]} |F_t^h|_{W^{m-2k-3}_p}^q
\leq Nh^{q(k+1)}\sum_{j=0}^k\E
\sup_{t\in[0,T]}|u^{(j)}_t|^q_{W_p^{m-2j}}
$$
$$
\leq Nh^{q(k+1)}(\E|\psi|_{W^{m}_p}^q
+
\E F_{m,p}^q(T)+\E G_{m,p}^q(T)), 
$$
where $N$ denotes some contstants which  
depend only on $K$, $m$, $d$, $q$, $p$, $T$  
and $|\Lambda|$. 
Putting these inequalities together we obtain the estimate
$$
\E\sup_{t\in[0,T]}
\sup_{x\in {\bG_h}}
|\overline{r}_t^h(x)|^q+\E\sup_{t\in[0,T]}
|\overline{r}_t^h|^q_{l_{p,h}}
$$
\begin{equation}                                                                 \label{estimate r0}
\leq Nh^{q(k+1)}(\E|\psi|_{W^{m}_p}^q
+
\E F_{m,p}^q(T)+\E G_{m,p}^q(T)), 
\end{equation}
for all $h>0$ with a constant $N=N(K, m, d, q, p, T,|\Lambda|)$. 
Thus we have the following theorem.

\begin{theorem}                                                                   \label{theorem main0}
Let $k\geq0$ be an integer and let Assumptions \ref{As0},  
\ref{As1} and \ref{As2} 
hold with $m>2k+3+d/p$. 
Then there are continuous random fields $u^{(1)},\ldots u^{(k)}$ 
on $[0,T]\times \R^d$, 
independent of $h$, such that almost surely
\begin{equation}                                                                       \label{expansion0}
u^h_t(x)=\sum_{j=0}^k\frac{h^j}{j!}u^{(j)}_t(x)+{\overline r}_t^h(x)
\end{equation}
for all $t\in[0,T]$ and $x\in \bR^d$, where $u^{(0)}=u$, 
$u^h$ is the solution of \eqref{differenceeq2}-\eqref{differenceeq2-initial}, 
and ${\overline r}^h$ 
is a continuous random field on $[0,T]\times \R^d$, which for any $q>0$ 
satisfies the estimate  \eqref{estimate r0}. 
\end{theorem}
\begin{proof}
By Theorems \ref{theorem 1}, \ref{theorem-Apr9} and \ref{theorem system}
$u^h$, $u^{(0)}$, $u^{(1)}$,...,$u^{(k)}$ are $W^{m-1}_p$-valued 
continuous processes. Since due to our assumption $m-1>d/p$, by 
Sobolev's theorem on embedding $W^m_p(\bR^d)$ into $C_b(\bR^d)$ 
we get that these random fields are continuous in $(t,x)\in [0,T]\times\bR^d$. 
(Remember that we always identify the functions with their 
continuous version when they have one.) Hence 
\eqref{expansion0} holds by the definition of $\overline{r}^h$, and 
estimate \eqref{estimate r0} is proved above. 
\end{proof} 

To finish the proof of Theorem \ref{Thm1} we need only show 
that if Assumption \ref{AsLambda} holds then 
under the conditions of Theorem \ref{theorem main0} 
the restriction of the solution $u^h$ 
of \eqref{differenceeq2}-\eqref{differenceeq2-initial} onto 
$[0,T]\times\bG_h$ is a continuous $l_p$-valued process which 
solves \eqref{differenceeq}-\eqref{differenceeq-initial}. 
To this end note that under the conditions of Theorem \ref{theorem main0} 
$u^h$ is a continuous $W^{m-1}_p$ valued process, 
and if Assumption \ref{AsLambda} also holds then by 
\eqref{embedding} its restriction to $[0,T]\times\bG_h$ is a continuous 
$l_p$-valued process. To see that this process satisfies 
\eqref{differenceeq}-\eqref{differenceeq-initial} we 
fix a point $x\in\bG_h$ 
and take a nonnegative smooth function $\varphi$ with 
compact support in $\bR^d$ such that its integral over $\bR^d$ 
is one. Define for each integer $n\geq1$ the function 
$\varphi^{(n)}(z)=n^{d}\varphi(n(z-x))$,  
$z\in\bR^d$. 
Then by Definition \ref{definition FDE} 
we have for $u^h$, the solution of 
\eqref{differenceeq}-\eqref{differenceeq-initial}, 
that almost surely 
$$
(u_t^h,\varphi^{(n)})=(\psi,\varphi^{(n)})
+\int_0^t(L^h_tu_s^h+f_s,\varphi^{(n)})\,ds
+\int_0^t(\nu^{r}_su_s^h+g^r_s,\varphi^{(n)})\,dw^r_s
$$
for all $t\in[0,T]$ and for all $n\geq1$. 
Letting here $n\to\infty$, for each $t\in[0,T]$ we get  
\begin{equation}                                                                   \label{pointwise}
u_t^h(x)=\psi(x)
+\int_0^t(L^h_su_s^h(x)+f_s(x))\,ds
+\int_0^t(\nu^{r}_tu_s^h(x)+g^r_s(x))\,dw^r_s
\end{equation}
almost surely, since $u^h$, $\psi$, $f$, $\nu$, $g$  and 
the coefficients of $L^h$ are continuous in $x$, due 
to Sobolev's theorem on embedding 
$W^m_p(\bR^d)$ into $C_b(\bR^d)$ in the case  $m>d/p$. 
Note that both $u^h_t(x)$ and the random field on 
the the right-hand side of equation \eqref{estimate r0} 
are continuous in $t\in[0,T]$.  
Therefore we have this 
equality almost surely for all $t\in[0,T]$ 
and $x\in\bG_h$. 
The proof of Theorem \ref{Thm1} is complete.

\medskip
The extrapolation result, Theorem \ref{Thm2}, follows from 
Theorem \ref{Thm1} by standard calculations, 
and hence Theorem \ref{Thm3} on  the rate of
almost sure convergence follows 
by a standard application of the Borel-Cantelli Lemma.

\qed

\emph{Proof of Theorem \ref{Thm4}.} 
Let 
$\rho(x)=\rho_{\overline{s}}(\epsilon x)
=1/(1+|\epsilon x|^2)^{\overline{s}/2}$, 
where $\epsilon>0$ is to be determined later and choose $p$ large enough so that 
$1>d/p$ - and therefore $m>2k+3+d/p$ -, and Assumption \ref{As2} holds 
for  $\psi\rho$, $f\rho$ and $g\rho$ in place of 
$\psi$, $f$ and $g$, respectively. 
After some calculations one gets that $u$ is the solution of 
\eqref{SPDE-poly}-\eqref{SPDE-poly-initial} if and only if 
$u\rho$ is the solution of the equation
$$
dv_t(x)
=(D_i\hat{a}_t^{ij}(x) D_{j}v_t(x)+\hat{b}_t^{i}(x)D_{i}v_t(x)+\hat{c}_t(x)v_t(x)+f_t\rho(x))\,dt
$$
\begin{equation}                                                                       \label{May3}
+(\nu^r_t(x)v_t(x)+g^r_t\rho(x))\,dw^r_t
\end{equation}
for $(t,x)\in[0,T]\times\R^d$, with the initial condition 
\begin{equation}\label{May3-initial}
v_0(x)=\psi\rho(x),
\end{equation} 
for $x\in\R^d$, where the coefficients are given by
$$\hat{a}^{ij}=a^{ij},$$
$$\hat{b}^i=b^i_t-2\sum_{j=1}^da^{ij}\frac{D_j\rho}{\rho} \text{  for  }i\neq0,$$
$$\hat{c}=c-\sum_{i,j=1}^da^{ij}\frac{D_iD_j\rho}{\rho}-\sum_{i,j=1}^dD_ia^{ij}\frac{D_j\rho}{\rho}-\sum_{i=1}^d\hat{b}^i\frac{D_i\rho}{\rho}.$$
Due to our choice of $\rho$, these coefficients still satisfy the conditions of Theorem 2.1 from \cite{GGK}. Applying this theorem, we obtain a $W^m_p$-valued uniqe solution $v$. Using Sobolev embedding, we get that $v/\rho$ - which is a solution of \eqref{SPDE-poly} - is a $P^{m-1}_p$-valued process. 

One can similarly transform the finite difference equations, using \eqref{Leibniz}-\eqref{Apr 18}. It turns out that $u^h$ is a solution 
of \eqref{differenceeq}-\eqref{differenceeq-initial} if and only
 if $u^h\rho$ is a solution of the equation
\begin{equation}                                                                                        \label{May3.0}
v^h_t(x)=\{\hat{L}_t^h(x)v_t^h(x)+f_t\rho(x))\,dt
+(\nu^r_t(x)v_t^h(x)+g^r_t\rho(x))\,dw_t^r
\end{equation}
for $(t,x)\in[0,T]\times \bG_h$ with initial condition 
\begin{equation}                                                                                \label{May3.0-initial}
v_0^h(x)=\psi\rho(x),
\end{equation}
for $x\in \bG_h$, where 
$$
\hat{L}^h_t\varphi
=\sum_{\lambda\in\Lambda_0}
\delta_{-h,\lambda}(
\hat{\mathfrak{a}}_h^{\lambda}\delta_{h,\lambda}\varphi)
+\sum_{\gamma\in\Lambda_1}\hat{\mathfrak{p}}_h^{\gamma}\delta_{h,\gamma}
\varphi
+\sum_{\gamma\in\Lambda_1}\hat{\mathfrak{c}}_h^{\gamma}T_{h,\gamma}
\varphi,
$$
with 
$$
\hat{\mathfrak{a}}^{\lambda}_h=\mathfrak{a}^{\lambda}_h,
$$
$$
\hat{\mathfrak{p}}^{\gamma}_h
=\mathfrak{p}^{\gamma}_h+\frac{(T_{h,-\lambda}\mathfrak{a}^{\lambda})\delta_{h,-\lambda}\rho-(T_{h,\lambda}\mathfrak{a}^{-\lambda})\delta_{h,\lambda}\rho}{\rho},
$$

$$
\hat{\mathfrak{c}}^{\lambda}_h
=\mathfrak{c}^{\lambda}_h\frac{\rho}{T_{h,\lambda}\rho}-\frac{(\delta_{h,-\lambda}\mathfrak{a}^{\lambda})\delta_{h,-\lambda}\rho-\mathfrak{a}^{\lambda}\delta_{h,-\lambda}\delta_{h,\lambda}\rho+\hat{\mathfrak{p}}^{\lambda}\delta_{h,\lambda}\rho}
{T_{h,\lambda}\rho},
$$
where $\mathfrak{a}^{\lambda}$ is understood to be 0 when not defined. 

As it was mentioned earlier, the restriction to 
$\bG_h$ of the continuous modifications of 
$\psi\rho,f\rho,g\rho$ are in $l_{p,h}$, $l_{p,h}$-valued, and 
$l_{p,h}(l_2)$-valued processes, respectively. 
The coefficients above are bounded, so as we have already seen, 
there exists a unique $l_{p,h}$-valued solution $v^h$, in particular, it is bounded. 
Therefore $v^h/\rho$ is a solution of \eqref{differenceeq} and has polynomial growth.

By choosing $\epsilon$ small enough, $|\delta_{h,\lambda}\rho/\rho|$ can be made arbitrarily small, uniformly in $x\in\R^d,\lambda\in\Lambda,|h|<1$. In particular, we can choose it to be small enough such that 
$\hat{\mathfrak{p}}_h^{\gamma}\geq 0$. 
Moreover, all of the smoothness and boundedness properties of the coefficients are preserved. Therefore \eqref{May3.0}-\eqref{May3.0-initial} is a finite difference scheme for the equation \eqref{May3}-\eqref{May3-initial} such that it satisfies Assumptions 
\ref{AsLambda} through \ref{As2}. Claims (iii) and (iv) then follow from applying Theorems \ref{Thm1} and \ref{Thm3}. 

\qed

\medskip
\textbf{Acknowledgement}
The authors would like to thank the anonymous referee for his remarks, 
which helped to improve the presentation of the paper.

\end{document}